\documentclass[11pt]{amsart}
\usepackage{amsmath,amssymb,amsbsy,amsfonts,amsthm,latexsym,
            amsopn,amstext,amsxtra,euscript,amscd,color,mathrsfs}

\PassOptionsToPackage{hyphens}{url}\usepackage{hyperref}
 
 \usepackage[capbesideposition=outside,capbesidesep=quad]{floatrow}

\restylefloat{table}
\restylefloat{table}
            
\usepackage{multirow,caption}
            
\usepackage{amscd}
\usepackage{color,enumerate}

\newcommand{\RNum}[1]{\lowercase\expandafter{\romannumeral #1\relax}}

\usepackage[colorinlistoftodos,prependcaption,textsize=tiny]{todonotes}

\newtheorem{thm}{Theorem}[section]
\newtheorem{lem}[thm]{Lemma}
\newtheorem{cor}[thm]{Corollary}

\newtheorem{thm-con}[thm]{Theorem-Conjecture}
\numberwithin{equation}{section}

\theoremstyle{definition}

\def\im{{\rm Im}}

\def\cB{{\mathcal B}}

\def\cW{{\mathcal W}}

\newcommand{\F}{\Bbb F}

\def\Tr{{\rm Tr}}

\begin{document}
\title[ The binary Gold function and its $c$-boomerang connectivity table]{The binary Gold function and its $c$-boomerang connectivity table}  
\author[S. U. Hasan]{Sartaj Ul Hasan}
\address{Department of Mathematics, Indian Institute of Technology Jammu, Jammu 181221, India}
\email{sartaj.hasan@iitjammu.ac.in}

\author[M. Pal]{Mohit Pal}
\address{Department of Mathematics, Indian Institute of Technology Jammu, Jammu 181221, India}
\email{2018RMA0021@iitjammu.ac.in}

\author[P.~St\u anic\u a]{Pantelimon~St\u anic\u a}
\address{Applied Mathematics Department, Naval Postgraduate School, Monterey 93943, USA}
\email{pstanica@nps.edu}

\begin{abstract}
Here, we give a complete description of the $c$-Boomerang Connectivity Table for the Gold function over finite fields of even characteristic, by using double Weil sums. In the process we generalize a result of Boura and Canteaut (IACR Trans. Symmetric Cryptol. 2018(3) : 290-310, 2018) for the classical boomerang uniformity.
\end{abstract}

\keywords{Finite fields, double Weil sums, boomerang uniformity, $c$-boomerang uniformity}
\subjclass[2010]{12E20, 11T24, 11T06, 94A60}
\maketitle 

\section{Introduction}
Let $\F_q$ be the finite field with $q=p^n$ elements, where $p$ is a prime and $n$ is a positive integer. The multiplicative cyclic group of nonzero elements of the finite field is denoted by $\F_q^* = \langle g \rangle$, where $g$ is a primitive element of $\F_q$. 
%The {\em canonical additive character} is a homomorphism $\chi_1 : \F_q \rightarrow \mathbb{C}$ of the additive group of $\F_q$ defined as follows
%$$\chi_1(x)=\exp\left(\frac{2\pi i\, {\rm Tr}(x)}{p} \right),$$
%where $\mathbb C$ is the field of complex numbers and ${\rm Tr} : \F_q \rightarrow \F_p$ is the absolute trace defined by $\Tr(x)=x+x^p+x^{p^2}+\cdots+x^{p^{n-1}}$  (to emphasize the dimension, we sometimes write this as $\Tr^n_1$). We define the relative trace $\Tr_e:\F_{p^n}\to\F_{p^e}, e|n$,  by $\Tr_e(x) =x+x^{p^e}+x^{p^{2e}}+\cdots+ x^{p^{e(\frac{n}{e}-1)}}$. 
%Note that all additive characters of $\F_q$ can be expressed in terms of $\chi_1$~\cite[Theorem 5.7]{LN97}.
%
%For each $0\leq k \leq q-2$, the $k$-th multiplicative character is a homomorphism $\psi_k : \F_q \rightarrow \mathbb{C}$ of the multiplicative group of $\F_q$ defined as follows
%$$\psi_k\left(g^\ell\right) =\exp\left(\frac{2\pi ik\ell}{q-1}\right) \quad \text{for}\; \ell =0, \ldots, q-2.$$
%It is well-known that the group of multiplicative characters of $\F_q$ is a cyclic group of order $q-1$ with identity element $\psi_0$~~\cite[Corollary 5.9]{LN97}.
%
%In the theory of finite fields, exponential sums  are important tools in the study of number of solutions of equations over finite fields. As a special case, the Gauss' sums  are defined as follows
%$$\displaystyle G(\psi,\chi)=\sum_{x\in\F_q^*} \psi(x)\chi(x),$$
%where $\chi$ and $\psi$ are additive and multiplicative characters of $\F_q$, respectively. A Weil sum is yet another important character sum   defined as follows
A Weil sum is an important character sum defined as follows

$$\sum_{x\in \F_q }\chi(F(x)),$$
where $\chi$ is an additive character of $\F_q$ and $F(x)$ is a polynomial in $\F_q[x]$. It is well-known that a polynomial $F(x)$ over finite field $\F_q$ is a permutation polynomial (PP) if and only if its Weil sum $\sum_{x\in \F_q }\chi(F(x))=0$ for all nontrivial additive characters $\chi$ of $\F_q.$ Permutation polynomials are a very important class of polynomials as they have applications in coding theory and cryptography, especially in the substitution boxes (S-boxes) of the block ciphers. The security of the S-boxes relies on certain properties of the function $F(x)$, e.g., its differential uniformity, boomerang uniformity, nonlinearity etc.

Recently, Cid et al.~\cite{cid} introduced a ``new tool" for analyzing the boomerang style attack proposed by Wagner~\cite{Wag}. This new tool is usually referred to as Boomerang Connectivity Table (BCT). Boura and Canteaut~\cite{BC} further studied BCT and coined the term boomerang uniformity, which is essentially the maximum value in the BCT. Li et al.~\cite{KL} provided new insights in the study of BCT and presented an equivalent technique to compute BCT, which does not require the compositional inverse of the permutation polynomial $F(x)$ at all. In fact, Li et al.~\cite{KL} also gave a characterization of BCT in terms of Walsh transform and gave a class of permutation polynomial with boomerang uniformity $4$.

Recently, St\u{a}nic\u {a} \cite{PSB} extended the notion of BCT and boomerang uniformity. In fact, he defined what he termed as $c$-BCT and $c$-boomerang uniformity for an arbitrary polynomial function $F$ over $\F_q$ and for any $c\neq0 \in \F_q$. Let $a,b \in \F_q$, then the entry of the $c$-Boomerang Connectivity Table ($c$-BCT) at $(a,b)\in \mathbb{F}_{p^n}\times \mathbb{F}_{p^n}$, denoted as $_c\cB_F(a,b)$, is the number of solutions in $\mathbb{F}_{p^n}\times \mathbb{F}_{p^n}$ of the following system
\begin{equation}\label{eqsys}
  \begin{cases} 
  F(x)-cF(y)=b \\
  F(x+a)-c^{-1}F(y+a)=b.
  \end{cases}
\end{equation} 

The $c$-boomerang uniformity of $F$ is defined as $$\beta_{F,c}= \max_{a,b \in \mathbb{F}_{p^n}^*}~  _c\cB_F(a,b).$$
In yet other recent papers, St\u{a}nic\u {a} ~ \cite{PSS,PSW} further studied the $c$-BCT for the swapped inverse function and also gave an elegant description of the $c$-BCT entries of the power map in terms of double Weil sums. He further simplified his expressions for the Gold function $x^{p^k+1}$ over $\F_{p^n}$, for all $1\leq k < n$ and $p$ odd. In this paper, we shall complement the work of~\cite{PSW} to the finite fields of even characteristic ($p=2$).

The paper is structured as follows. Section~\ref{S2} contains some preliminary results that will be used across the sections. Section~\ref{S3} contains the characterization of $c$-BCT entries in terms of double Weil sums. For $c=1$, we further simplify this expression in Section~\ref{S4}. In fact, Theorem~\ref{generalization} generalizes previously known results of Boura and Canteaut~\cite{BC}. In Section~\ref{S5}, we consider the case when $c \in \F_{2^e}\backslash \{0,1\}$, where $e=\gcd(k,n)$. In Section~\ref{S6}, we discuss the general case. Finally, in Section~\ref{S7}, we discuss the affine, extended affine and CCZ-equivalence as it relates to c-boomerang uniformity.

\section{Preliminaries}\label{S2}

We begin this section by first recalling the recent notion of $c$-differentials introduced in~\cite{CDU}.
We shall assume that $q=2^n$ for rest of the paper. For an $(n,n)$-function $F: \F_{q} \to \F_{q}$, and $c\in \F_{q}$, we define the ({\em multiplicative}) {\em $c$-derivative} of $F$ with respect to~$a \in \F_q$ to be the  function
\[
 _cD_{a}F(x) =  F(x + a)- cF(x), \mbox{ for  all }  x \in \F_q.
\]
Further, for $a,b\in\F_q$, we let the entries of the $c$-Difference Distribution Table ($c$-DDT) be defined by ${_c\Delta}_F(a,b)= \#{\{x\in\F_q : F(x+a)-cF(x)=b\}}$. We call the quantity
\[
\delta_{F,c}=\max\left\{{_c\Delta}_F(a,b)\,|\, a,b\in \F_q, \text{ and } a\neq 0 \text{ if $c=1$} \right\},\]
the {\em $c$-differential uniformity} of~$F$. Note that the case $c=1$ corresponds to the usual notion of differential uniformity. The interested reader may refer to \cite{BC1,HPRS,YMZ, SG20,SRT,ZH} for some recent results concerning $c$-differential uniformity. 

The following theorem is a ``binary" analogue of~\cite[Theorem 1]{PSW}, which gives a nice connection between $c$-BCT and $c$-DDT entries of the power map $x^d$ over $\F_{2^n}$. 
\begin{thm}
\label{even1}
Let $F(x)=x^d$ be a power function on $\F_q$, $q=2^n$ and $c\in\F_q^*$. Then, for fixed $b\in \F_{q}^*$, the $c$-Boomerang Connectivity Table entry  $_c\cB_F(1,b)$ at $(1,b)$ is given by 
{\small
\[
\frac{1}{q}\left(\sum_{w \in\F_q} ({_c\Delta}_F (w,b)+ {_{c^{-1}}\Delta}_F (w,b))\right)-1 +\frac1{q^2}\sum_{\alpha,\beta\in\F_q,\alpha\beta\neq 0} \chi_1(b(\alpha+\beta))\, S_{\alpha,\beta}\, S_{\alpha c, \beta c^{-1}},
\]
}
with
\begin{align*}
S_{\alpha,\beta}&=\sum_{x\in \F_q} \chi_1\left(\alpha x^d\right)\chi_1\left(\beta (x+1)^d\right)\\
&= \frac{1}{(q-1)^2} \sum_{j,k=0}^{q-2}   G(\bar\psi_j,\chi_1) G(\bar\psi_k,\chi_1) \sum_{x\in \F_q} \psi_1\left((\alpha x^d)^j (\beta (x+1)^d)^k\right),
\end{align*}
where $\chi_1$ is the canonical additive character of  the additive group of $\F_q$, $\psi_k$ is the $k$-th multiplicative character of the multiplicative group of $\F_q$ and $G(\psi,\chi)$ is the Gauss sum.
\end{thm}

We shall now state some lemmas that will be used in the sequel. The following lemma is well-known and has been used in various contexts.

\begin{lem}\label{gcd}
Let $e=\gcd(k,n)$. Then 
 \begin{equation*}
  \gcd(2^k+1, 2^n-1)=
  \begin{cases}
   1 & ~\mbox{if}~ n/e~\mbox{is odd,}\\
   2^e+1 & ~\mbox{if}~ n/e~\mbox{is even.}
  \end{cases}
 \end{equation*}
\end{lem}
We shall also use the following lemma, which appeared in \cite{RSC}, describing the number of roots in $\F_{2^n}$ of a linearized polynomial $u^{2^k}x^{2^{2k}}+ux$, where $u\in \F_{2^n}^*.$
\begin{lem} \cite[Theorem 3.1]{RSC} \label{pp}
Let $g$ be a primitive element of $\F_{2^n}$ and let $e=\gcd(n,k)$. For any $u\in \F_{2^n}^*$ consider the linearized polynomial $L_u(x)= u^{2^k}x^{2^{2k}}+ux$ over $\F_{2^n}.$ Then for the equation $L_u(x)=0$, the following are true:
\begin{enumerate}
\item[$(1)$] If $n/e$ is odd, then there are $2^e$ solutions to this equation for any choice of $u \in \F_{2^n}^*$;
\item[$(2)$] If $n/e$ is even and $u=g^{t(2^e+1)}$ for some $t$, then there are $2^{2e}$ solutions to the equation;
\item[$(3)$] If $n/e$ is even and $u\neq g^{t(2^e+1)}$ for any $t$, then $x=0$ is the only solution.
\end{enumerate}
\end{lem}
The explicit expression for the Weil sum of the form $\sum_{x\in \F_{2^n}}\chi_1(ux^{2^k+1}+vx)$, where $u,v \in \F_{2^n}$, is obtained in \cite{RSC}. In what follows, we shall denote the Weil sum $\sum_{x\in \F_q}\chi(ux^{2^k+1}+vx)$ by $\mathfrak{S}(u,v)$. The following lemma gives the explicit expression for $\mathfrak{S}(u,0).$
\begin{lem}\cite{RSC} \label{zero}
 Let $\chi$ be any nontrivial additive character of $\F_{q}$ and $g$ be the primitive element of the cyclic group $\F_q^*$. The following hold:
 \begin{enumerate}
  \item[$(1)$] If $n/e$ is odd, then 
  \begin{equation*}
   \sum_{x\in \F_{q}} \chi(ux^{2^k+1}) =
   \begin{cases}
    q &~\mbox{if}~ u=0,\\
    0 &~\mbox{otherwise.}
   \end{cases}
  \end{equation*}
  \item[$(2)$] Let $n/e$ be even so that $n=2m$ for some integer $m$. Then
  \begin{equation*}
   \displaystyle \sum_{x\in \F_{q}} \chi(ux^{2^k+1}) =
   \begin{cases}
  \displaystyle  (-1)^{m/e}2^m &~\mbox{if}~ u\neq g^{t(2^e+1)}~\mbox{for any integer}~t,\\
    \displaystyle (-1)^{\frac{m}{e}+1}2^{m+e} &~\mbox{if}~ u= g^{t(2^e+1)}~\mbox{for some integer}~t.
   \end{cases}
  \end{equation*}  
 \end{enumerate}
\end{lem}
From Lemma \ref{gcd}, it is easy to see that when $n/e$ is odd, the power map $x^{2^k+1}$ permutes~$\F_{2^n}$. Therefore if $u \neq0$, there exists a unique element $\gamma \in \F_{q}^*$ such that $\gamma^{2^k+1}=u$ and hence
\begin{align*}
   \mathfrak{S}(u,v) 
   &=\sum_{x\in \F_q}\chi(ux^{2^k+1}+vx)\\
   &=\sum_{x\in \F_q}\chi(x^{2^k+1}+v\gamma^{-1}x)\\
   &= \mathfrak{S}(1, v\gamma^{-1}).
\end{align*}
The following lemma gives the expression for the Weil sum $\mathfrak{S}(1,v)$ for $v\neq0$ and $n/e$ odd.
\begin{lem}\cite[Theorem 4.2]{RSC} \label{odd}
Let $v\neq0$ and $n/e$ is odd. Then
\begin{equation*}
\mathfrak{S}(1,v) =
\begin{cases}
0 &~\mbox{if}~ {\rm Tr}_e(v) \neq 1,\\
\displaystyle \left(\frac{2}{n/e}\right)^e\, 2^{\frac{n+e}{2}} &~\mbox{if}~ {\rm Tr}_e(v) = 1,
\end{cases}
\end{equation*} 
where $\displaystyle \left(\frac{2}{n/e}\right)$ is the Jacobi symbol.
\end{lem}
In the case when $u,v\neq0$ and $n/e$ is even, the Weil sum $\mathfrak{S}(u,v)$ depends whether or not the linearized polynomial $L_u(x)=u^{2^k}x^{2^{2k}}+ux$ is a permutation of $\F_{2^n}.$ The following lemma gives the expression for Weil sum $\mathfrak{S}(u,v)$ for $u,v\neq0$ and $n/e$ even.

\begin{lem} \cite[Theorem 5.3]{RSC}\label{even}
 Let $u,v \in \F_{q}^*$ and $n/e$ is even so $n=2m$ for some integer $m$. Then
 \begin{enumerate}
  \item[$(1)$] If $u\neq g^{t(2^e+1)}$ for any integer $t$ then $L_u$ is a PP. Let $x_u \in \F_q$ be the unique solution of the equation $L_u(x)=v^{2^k}$. Then 
  $$\displaystyle \mathfrak{S}(u,v)= (-1)^{m/e}2^m \chi_1(u{x_u}^{2^k+1}).$$
  \item[$(2)$] If $u= g^{t(2^e+1)}$ for some integer $t$, then $\mathfrak{S}(u,v)=0$ unless the equation $L_u(x)=v^{2^k}$ is solvable. If the equation $L_u(x)=v^{2^k}$ is solvable with some solution, say $x_u$, then 
  \begin{equation*}
   \mathfrak{S}(u,v)=
   \begin{cases}
    \displaystyle (-1)^{m/e}2^m \chi_1(u{x_u}^{2^k+1}) &~\mbox{if}~ {\rm Tr}_e(u)\neq0,\\
    \displaystyle (-1)^{\frac{m}{e}+1}2^{m+e} \chi_1(u{x_u}^{2^k+1}) &~\mbox{if}~ {\rm Tr}_e(u)= 0. \\
   \end{cases}
  \end{equation*}
 \end{enumerate}
\end{lem}

\section{The binary Gold function}
\label{S3}
In this section, we shall give the explicit expression for the $c$-BCT entries of the Gold function $x^{2^k+1}$ over $\F_{2^n}$, for all $c\neq0$. Recall that the $c$-boomerang uniformity of a power function $F(x)=x^d$ over $\F_{2^n}$ is given by $\displaystyle \max_{b \in \mathbb{F}_{2^n}^*}~ _c\cB_F(1,b) $, where $_c\cB_F(1,b)$ is the number of solutions in $\mathbb{F}_{q}\times \mathbb{F}_{q}$, $q=2^n$ of the following system
\begin{equation}
\label{eq:eq3.1}
\begin{cases}
x^d+cy^d=b \\
(x+1)^d+c^{-1}(y+1)^d=b.
\end{cases}
\end{equation}
As done in~\cite{PSW},  for $b \neq 0$ and fixed $c\neq 0$, the number of solutions  $(x,y)\in\F_q^2$ of the system ~\eqref{eq:eq3.1} is given by
\allowdisplaybreaks
\begin{align*}
{_c\cB_F}(1,b)
&=\frac{1}{q^2}\sum_{x,y\in \F_q}\sum_{\alpha \in\F_q} \chi_1\left(\alpha\left(x^d+cy^d+b \right) \right) \sum_{\beta\in\F_q}\chi_1\left(\beta\left((x+1)^d+c^{-1} (y+1)^d+b \right) \right)\\
&=\frac{1}{q^2}\sum_{\alpha,\beta \in \F_q}\chi_1\left(b\left(\alpha+\beta \right) \right) \sum_{x \in\F_q} \chi_1\left(\alpha x^d+\beta(x+1)^d \right) \sum_{y \in\F_q} \chi_1\left(c\alpha y^d+c^{-1}\beta (y+1)^d \right)\\
&=\frac{1}{q^2}\sum_{\alpha,\beta \in \F_q}\chi_1\left(b\left(\alpha+\beta \right) \right) S_{\alpha, \beta} S_{c\alpha,  c^{-1}\beta},
\end{align*}
where $S_{\alpha,\beta} =\sum_{x\in \F_q} \chi_1\left(\alpha x^d+\beta (x+1)^d\right)$. Therefore, the problem of computing the $c$-BCT entry ${_c\cB_F}(1,b)$ is reduced to the computation of the product of the Weil sums $S_{\alpha, \beta}$ and $S_{c\alpha, c^{-1}\beta}.$ Now, in the particular case when $d=2^k+1$, i.e., for the Gold case, we shall further simplify the expression for $S_{\alpha, \beta}$ as follows:
\allowdisplaybreaks
\begin{align*}
S_{\alpha,\beta}&=\sum_{x\in \F_q} \chi_1\left(\alpha x^{2^k+1}+\beta (x+1)^{2^k+1} \right)\\
&= \chi_1(\beta)\sum_{x\in \F_q} \chi_1((\alpha+\beta) x^{2^k+1})~ \chi_1 (\beta x^{2^k}+\beta x) \\
&= \chi_1(\beta)\sum_{x\in \F_q} \chi_1((\alpha+\beta) x^{2^k+1})~ \chi_1 ((\beta^{2^{n-k}}x)^{2^k}+\beta x) \\
&= \chi_1(\beta)\sum_{x\in \F_q} \chi_1((\alpha+\beta) x^{2^k+1})~ \chi_1 ((\beta^{2^{n-k}}+\beta)x) \\
&= \chi_1(\beta)\sum_{x\in \F_q} \chi_1((\alpha+\beta) x^{2^k+1}+(\beta^{2^{n-k}}+\beta)x) \\
&= \chi_1(\beta)\sum_{x\in \F_q} \chi_1(A x^{2^k+1}+Bx),
\end{align*}
where $A=\alpha+\beta$ and $B=\beta^{2^{n-k}}+\beta$. Here one may note that $A=0$ if and only if $\alpha=\beta$. Also, $B=0$ if and only if $\beta \in \F_{2^e}$, since 
\begin{align*}
 B = 0 
 &\Leftrightarrow \beta^{2^{n-k}}=\beta\\
 &\Leftrightarrow \beta^{2^{n-k}-1}=1\\
 &\Leftrightarrow \beta^{2^{\gcd(n-k,n)}-1}=1\\
 &\Leftrightarrow \beta^{2^{e}-1}=1,~(\mbox{as}~ \gcd(n-k,n)=e)\\
 &\Leftrightarrow \beta \in \F_{2^e}.
\end{align*}
Now we shall calculate $S_{\alpha, \beta}$ in two cases, namely, $n/e$ odd and $n/e$ even, respectively.\\
\textbf{Case 1}: $n/e$ is odd.\\ 
In this case, if $\alpha=\beta$ and $\beta \in \F_{2^e}$, then $S_{\alpha, \beta}= q\chi_1(\beta)$. If $\alpha=\beta$ and $\beta \in \F_{2^n} \backslash \F_{2^e}$ then $S_{\alpha, \beta}= 0$. In the event of $\alpha \neq \beta$ and $\beta \in \F_{2^e}$, again we have $S_{\alpha, \beta}= 0$. Finally, if $\alpha \neq \beta$ and $\beta \in \F_{2^n} \backslash \F_{2^e}$, by Lemma~\ref{odd} we have,
\begin{equation*}
S_{\alpha, \beta}=
\begin{cases}
0 &~\mbox{if}~ {\rm Tr}_e(B\gamma^{-1})\neq 1, \\
\left(\frac{2}{n/e}\right)^{e} 2^{\frac{n+e}{2}} \chi_1(\beta) &~\mbox{if}~ {\rm Tr}_e(B\gamma^{-1})= 1,
\end{cases}
\end{equation*}
where $\gamma \in \F_{q}$ is the unique element such that $\gamma^{2^k+1}= A$. 

\noindent
\textbf{Case 2}: $n/e$ is even.\\
Let $n=2m$, for some positive integer $m$ and $g$ be a primitive element of the finite field $\F_q$. When $\alpha=\beta$ and $\beta \in \F_{2^e}$ then $S_{\alpha, \beta}= q\chi_1(\beta)$. If $\alpha=\beta$ and $\beta \in \F_{2^n} \backslash \F_{2^e}$ then again $S_{\alpha, \beta}= 0$. In the event of $\alpha \neq \beta$ and $\beta \in \F_{2^e}$, by Lemma~\ref{zero} we have 
\begin{equation*}
S_{\alpha, \beta}=
\begin{cases}
(-1)^{m/e} 2^m \chi_1(\beta) &~\mbox{if}~ A \neq g^{t(2^e+1)}~\mbox{for any integer}~ t, \\
(-1)^{\frac{m}{e}+1} 2^{m+e} \chi_1(\beta) &~\mbox{if}~ A = g^{t(2^e+1)}~\mbox{for some integer}~ t.
\end{cases}
\end{equation*}
Finally, when $\alpha \neq \beta$ and $\beta \in \F_{2^n} \backslash \F_{2^e}$, we shall consider two cases depending on whether or not the linearized polynomial $L_{A}(x)= A^{2^k}x^{2^{2k}}+Ax$ is a permutation polynomial. From Lemma~\ref{pp}, $L_{A}$ is a permutation polynomial if and only if $n/e$ is even and $A \neq g^{t(2^e+1)}$ for any integer $t$. Therefore, when $n/e$ is even and $A \neq g^{t(2^e+1)}$ for any integer $t$, the equation $L_{A}(x)= B^{2^k}$ will have a unique solution, say $x_{A}$. Therefore, by Lemma~\ref{even}, we have 
$$\displaystyle
S_{\alpha, \beta}= (-1)^{m/e} 2^m \chi_1(\beta) \chi_1(Ax_{A}^{2^k+1}).$$
Now if the linearized polynomial $L_{A}$ is not permutation, i.e, $n/e$ is even and $A = g^{t(2^e+1)}$ for some integer $t$, we again have two cases depending on whether or not the equation $L_{A}(x)= B^{2^k}$ is solvable. In the case when equation $L_{A}(x)= B^{2^k}$ is solvable, let $x_{A}$ be one of its solution. Therefore, by Lemma~\ref{even} we have,
\begin{equation*}
S_{\alpha, \beta}=
\begin{cases}
(-1)^{\frac{m}{e}+1} 2^{m+e} \chi_1(\beta) \chi_1\left(Ax_{A}^{2^k+1}\right) &~\mbox{if}~ {\rm Tr}_e(A)=0, \\
(-1)^{\frac{m}{e}} 2^{m} \chi_1(\beta) \chi_1\left(Ax_{A}^{2^k+1}\right) &~\mbox{if}~ {\rm Tr}_e(A)\neq 0.
\end{cases}
\end{equation*}
If $L_{A}(x)= B^{2^k}$ is not solvable, again, by Lemma~\ref{even}, $S_{\alpha, \beta}=0$. \\
Thus we have  computed $S_{\alpha, \beta}$ in all possible cases. Similarly, we can find $S_{c\alpha, c^{-1}\beta}$ by putting $c\alpha $ and $c^{-1} \beta$ in place of $\alpha$ and $\beta$, respectively. We shall now explicitly compute the $c$-BCT entry $_cB_F(1,b)$ for $c=1$, $c \in \F_{2^e} \backslash \{0,1\}$ and $c \in \F_{2^n}\backslash \F_{2^e}$ in the forthcoming sections.

\section{The case $c=1$} \label{S4}

When $c=1$, $S_{\alpha,\beta}$ and $S_{c\alpha, c^{-1}\beta}$ coincide, therefore for any fixed $b\neq0$, the $c$-BCT entry is given by,
\[_1\cB_F(1,b) =\frac{1}{q^2}\sum_{\alpha,\beta \in \F_q}\chi_1\left(b\left(\alpha+\beta \right) \right) S_{\alpha, \beta}^2.  \]
Let us denote $ \displaystyle T_b =  S_{\alpha,\beta}^2$. Now we shall consider two cases, namely, $n/e$ odd and $n/e$ even, respectively.\\
\\
\textbf{Case 1}: $n/e$ is odd. We consider the following subcases.

\begin{enumerate}
 \item If $\alpha = \beta$ and $\beta \in  \F_{2^e}$, then
\[ T_b^{[1]} = q^2\,\chi_1(\beta)^2 = q^2. \]

\item If $\alpha = \beta$ and $\beta \in \F_{2^n} \backslash \F_{2^e}$, then 
\[T_b^{[2]} = 0.\]

\item 
If $\alpha \neq \beta$ and $\beta \in \F_{2^e}$, then 
\[T_b^{[3]} = 0.\]

\item 
If $\alpha \neq \beta$ and $\beta \in \F_{2^n} \backslash \F_{2^e}$ then 
\begin{equation*}
T_b^{[4]} =
\begin{cases}
0 &~\mbox{if}~ {\rm Tr}_e(B\gamma^{-1})\neq 1, \\
\displaystyle
2^{n+e}  &~\mbox{if}~ {\rm Tr}_e(B\gamma^{-1})= 1.
\end{cases}
\end{equation*}
%where $\gamma \in \F_{q}$ is the unique element such that $\gamma^{2^k+1}= A.$
\end{enumerate}

Nyberg~\cite[Proposition 3]{NY} showed that the differential uniformity of the Gold function $x\mapsto x^{2^k+1}$ over $\F_{2^n}$ is $2^e$, where $e=\gcd(k,n)$. Also, from~\cite{cid}, we know that the boomerang uniformity of the APN function equals~$2$.    Boura and Canteaut~\cite[Proposition 8]{BC} proved that when $n/e$ is odd and $n\equiv2 \pmod 4$, then the differential as well as the boomerang uniformity of the Gold function $x\mapsto x^{2^k+1}$ is $4$. 
Our first theorem in this section generalizes the two previously mentioned results, and gives the boomerang uniformity of the Gold function for any parameters, when $\frac{n}{e}$ is odd. Note that we would require the notion of  Walsh-Hadamard transform in the proof of this theorem, which is defined as follows.

For $f:\F_{2^n}\to \F_2$ we define the {\it Walsh-Hadamard transform} to be the integer-valued function
\[
\displaystyle
 \cW_f(u)  = \sum_{x\in \F_{2^n}}(-1)^{f(x)+\Tr(u x)}, u \in \mathbb{F}_{2^n}.
\]
 The Walsh transform $ \cW_{F}(a,b)$ of an $(n,m)$-function $F: \F_{2^n} \to \F_{2^m}$ at $a\in \F_{2^n}, b\in \F_{2^m}$ is the Walsh-Hadamard transform of its component function ${\rm Tr}^m_1(bF(x))$ at $a$, that is,
\[
  \cW_{F}(a,b)=\sum_{x\in\F_{2^n}} (-1)^{\Tr^m_1(bF(x))-\Tr^n_1(ax)}.
\]

\begin{thm} \label{generalization}
 Let $F(x)=x^{2^k+1}$, $1\leq k < n$, be a function on $\F_{q}$, $q=2^n$, $n\geq2$. Let $c=1$ and $n/e$ be odd, where $e=\gcd(k,n)$. Then the $c$-BCT entry $_1\cB_F(1,b)$ of $F$ at $(1,b)$ is   
 \begin{equation*}
 _1\cB_F(1,b)= 0,~\text{or},~2^e,
\end{equation*}
 if $\Tr_e\left(b^{\frac{1}{2}}\right)=0$, respectively, $\Tr_e\left(b^{\frac{1}{2}}\right) \neq 0$.
\end{thm}
\begin{proof}
For every $\alpha,\beta$, let $A=\alpha+\beta,B=\beta^{2^{-k}}+\beta$, and
 $\gamma \in \F_{q}$ be the unique element such that $\gamma^{2^k+1}= A$. Further, let
 \begin{align*}
  \mathcal{A} &=\{ (\alpha, \beta) \in \F_q^2 \mid \alpha=\beta \in \F_{2^e}\},\\
 \mathcal{B} &=\{ (\alpha, \beta) \in \F_q^2 \mid \alpha=\beta \in \F_q \backslash \F_{2^e}\},\\
 \mathcal{C}&=\{ (\alpha, \beta) \in \F_q^2 \mid \alpha\neq\beta~and~ \beta \in \F_{2^e}\},\\
 \mathcal{D}&=\{ (\alpha, \beta) \in \F_q^2 \mid \alpha\neq\beta~and~ \beta \in \F_q \backslash \F_{2^e}\},\\
 \mathcal{E}&=\{ (\alpha, \beta) \in \mathcal{D} \mid {\rm Tr}_e(B\gamma^{-1})\neq 1\},\\
 \mathcal{F}&= \{(\alpha, \beta) \in \mathcal{D} \mid  {\rm Tr}_e(B \gamma^{-1})= 1 \}.
 \end{align*}
Then,
\allowdisplaybreaks
\begin{align*}
 _1\cB_F(1,b)=& \frac{1}{q^2} \left( \sum_{(\alpha, \beta) \in \mathcal{A}}  \chi_1(b(\alpha+\beta))T_b^{[1]} + \sum_{(\alpha, \beta) \in \mathcal{B}}  \chi_1(b(\alpha+
 \beta))T_b^{[2]}  +\sum_{(\alpha, \beta) \in \mathcal{C}}  \chi_1(b(\alpha+\beta))T_b^{[3]}+\right. \\
& \left. \sum_{(\alpha, \beta) \in \mathcal{E}}  \chi_1(b(\alpha+\beta))T_b^{[4]} +\sum_{\alpha, \beta \in \mathcal{F}}  \chi_1(b(\alpha+\beta))T_b^{[4]} \right) \\
 =& \frac{1}{q^2} \left( \sum_{(\alpha, \beta) \in \mathcal{A}} q^2 +\sum_{(\alpha, \beta) \in \mathcal{F}}  \chi_1(b(\alpha+\beta))2^{n+e}\right)\\
 =& 2^e+ \frac{2^e}{2^n} \sum_{(\alpha, \beta) \in \mathcal{F}}  \chi_1(b(\alpha+\beta)).
  \end{align*}

As  customary, $t^{-1}=t^{2^n-2}$, rendering $0^{-1}=0$. For each $\beta\in \F_{2^n}\setminus \F_{2^e}$, we let  (if $\beta\in\F_{2^e}$, $Y_\beta=\F_{2^n}$)
 \begin{align*}
 Y_\beta&=\left\{\gamma^{-1} \in\F_{2^n}\,:\, \Tr_e\left( (\beta^{2^{-k}}+\beta)\gamma^{-1}\right)=1 \right\},
 \end{align*}
 and
 \begin{align*}
 T_\beta&=\left\{d\in\F_{2^n}\,:\, \Tr_e((\beta^{2^{-k}}+\beta)d)=0  \right\}=\langle\beta^{2^{-k}}+\beta \rangle^{\perp_e}.
 \end{align*}
We shall use below  that when $\frac{n}{e}$ is odd, then $\Tr_e(1)=1$.
We label by $\langle S\rangle_e$ the $\F_{2^e}$-linear subspace in $\F_{2^n}$ generate by $S$ and we write $S^{\perp_e}$, for the trace orthogonal (via the relative trace $\Tr_e$) of the subspace $\langle S\rangle_e$ (if $e=1$, we drop the subscripts).
 Since $\Tr_e(1)=1$, then, $(\beta^{2^{-k}}+\beta)^{-1}\in Y_\beta$.
 If $\gamma_1^{-1},\gamma_2^{-1}\in Y_\beta$, then  $\gamma_1^{-1}+\gamma_2^{-1}\in T_\beta$, of cardinality $|T_\beta|=2^{n-1}$. Reciprocally, if $\gamma^{-1}\in Y_\beta$ and $d\in T_\beta$, it is easy to see that $\gamma^{-1}+d\in Y_\beta$. 
 Therefore, $Y_\beta$ is the affine subspace $Y_\beta=\gamma_\beta+T_\beta$, where $\gamma_\beta=(\beta^{2^{-k}}+\beta)^{-1}$.

Next, we observe that the  kernel of $\phi:\beta\mapsto \beta^{2^{-k}}+\beta$, say $\ker(\phi)$, is an $\F_2$-linear space of dimension $e$ (in fact, it is exactly $\F_{2^e}$) and the image of $\phi$, say $\im(\phi)$, is an $\F_2$-linear space  of dimension ${n-e}$. 
Further, we show that $\im(\phi)^{\perp_e}=\ker(\phi)$. We use below the fact that $\Tr_e(x^{2^e})=\Tr_e(x)$ and $e\,|\,k$. Let $u\in \im(\phi)^{\perp_e}$, that is, for all $\beta\in\F_{2^n}$,
  \begin{align*}
0&=\Tr_e(u(\beta^{2^{-k}}+\beta))=\Tr_e(u \beta^{2^{-k}}) +\Tr_e(u\beta)=\Tr_e(u^{2^k}\beta)+\Tr_e(u\beta)=\Tr_e((u+u^{2^k})\beta),
\end{align*}
and so, $u^{2^k}+u=0$, which shows the claim. For easy referral, if we speak of the dimension of an $F_{2^e}$-linear space $S$, we shall be using the notation $\dim_e S$ (no subscript if $e=1$).
 
 We will be using below the Poisson summation formula (see~\cite[Corollary 8.9]{CH1} and \cite[Theorem 2.15]{CS17}), 
 which states that if $f:\F_{2^n}\to \mathbb{R}$ and $S$ is a subspace of $\F_{2^n}$ of dimension $\dim S$, then
  \[
 \sum_{u\in \alpha+S} \cW_f(u) (-1)^{\Tr(\beta u)} =2^{\dim S} (-1)^{\Tr(\alpha\beta)} \sum_{u\in \beta+S^\perp} f(u) (-1)^{\Tr(\alpha u)},
 \]
and in particular,
 \[
 \sum_{u\in S} \cW_f(u)=2^{\dim S}\sum_{u\in S^\perp} f(u).
 \]

 Now, we are able to compute our sum (labelling $\alpha=\beta+\gamma^{2^k+1}$,  and writing $\phi^{-1}(t)=\{\beta\,:\, \phi(\beta)=t\}$; we also note that when $\frac{n}{e}$ is odd, $\gcd(2^k+1,2^n-1)=1$, and so $\gamma\mapsto \gamma^{2^k+1}$ is a permutation)
  \begin{align*}
& _1\cB_F(1,b) =2^e+ \frac{2^e}{2^n} \sum_{\substack{\beta\in \F_{2^n}\backslash \F_{2^e},\gamma\in \F_{2^n}\\   \Tr_e\left((\beta^{2^{-k}}+\beta)\gamma^{-1} \right)=1}}\chi_1\left(b \gamma^{2^k+1}\right)\\
 &\quad=2^e+ \frac{2^e}{2^n} \sum_{\beta\in \F_{2^n}\backslash \F_{2^e}}  \sum_{ \gamma^{-1} \in Y_\beta} \chi_1\left(b \gamma^{2^k+1}\right)\\
 &\quad=2^e+ \frac{2^{e}}{2^n} \sum_{\beta\in \F_{2^n}}  \sum_{x\in(\beta^{2^{-k}}+\beta)^{-1}+\langle\beta^{2^{-k}}+\beta \rangle^{\perp_e}} \chi_1\left(b x^{-2^k-1}\right)\\
  & \qquad\qquad \text{ \small (we used here that $Y_\beta=(\beta^{2^{-k}}+\beta)^{-1}+T_\beta$; we also added}\\
& \qquad\qquad\qquad\text{\small $\beta \in \F_{2^e}$, as it contributes 0 to the inner sum)}\\
&\quad=2^e+ \frac{2^{e}}{2^n} \sum_{\beta\in \F_{2^n}}  2^{-\dim S} \sum_{u \in\left(\langle\beta^{2^{-k}}+\beta \rangle^{\perp_e}\right)^\perp}   \cW_{g_\beta}(u) (-1)^{\Tr\left(u(\beta^{2^{-k}}+\beta)^{-1}\right)}\\
& \quad \text{ \small (by Poisson summation with $S^\perp= \langle  \beta^{2^{-k}}+\beta \rangle^{\perp_e}$, and $g_\beta(x)=\chi_1\big(b x^{-2^k-1}\big)$)}.
\end{align*}
We now analyze the $\F_2$-linear space 
\[
\left(\langle\beta^{2^{-k}}+\beta \rangle^{\perp_e}\right)^\perp
=\{x \in \F_{2^n}\,:\, \Tr(d x)=0, \forall d \text{ with } \Tr_e(d(\beta^{2^{-k}}+\beta))=0\}.
\]

Further, $\F_{2^n}$ has dimension $n/e$ as an $\F_{2^e}$-linear space and so,
$\dim_e \langle\beta^{2^{-k}}+\beta \rangle^{\perp_e}=\frac{n}{e}-1$ as an $\F_{2^e}$-linear space, and since $\F_{2^e}$ has dimension $e$ as an $\F_2$-linear space, then $\dim  \langle\beta^{2^{-k}}+\beta \rangle^{\perp_e}=n-e$ as an $\F_2$-linear space. Thus, $\dim \left( \langle\beta^{2^{-k}}+\beta \rangle^{\perp_e}\right)^\perp=e$. Moreover, $\Tr_e(\beta^{2^{-k}}+\beta)=0$ and if $u\in\F_{2^e}$  then $\Tr_e(u(\beta^{2^{-k}}+\beta)=u\Tr_e(\beta^{2^{-k}}+\beta)=0$, and consequently (since the dimensions match and $(\beta^{2^{-k}}+\beta)\F_{2^e}\subseteq S$)
\[
S=\left( \langle\beta^{2^{-k}}+\beta \rangle^{\perp_e}\right)^\perp=(\beta^{2^{-k}}+\beta)\F_{2^e}.
\]

We are now ready to continue the computation, thus,
\allowdisplaybreaks
  \begin{align*}
&_1\cB_F(1,b) 
=2^e+ \frac{2^{e}}{2^n} 2^{-e}  \sum_{\beta\in \F_{2^n}}  \sum_{u \in (\beta^{2^{-k}}+\beta)\F_{2^e}}   \cW_{g_\beta}(u)(-1)^{\Tr\left(u(\beta^{2^{-k}}+\beta)^{-1}\right)}\\
&=2^e+ \frac{2^{e}}{2^n} 2^{-e}  \sum_{\beta\in \F_{2^n}}  \sum_{d' \in \F_{2^e}}   \cW_{g_\beta}(d'(\beta^{2^{-k}}+\beta))(-1)^{\Tr(d')}\\
&=2^e+ \frac{2^{e}}{2^n} 2^{-e}  \sum_{\beta\in \F_{2^n}}  \sum_{d' \in \F_{2^e}}
\sum_{x\in\F_{2^n}} \chi_1\left(b x^{-2^k-1}+d'x(\beta^{2^{-k}}+\beta)+d' \right)\\
&=2^e+ \frac{2^{e}}{2^n} 2^{-e}   \sum_{d'\in \F_{2^e}} \sum_{x\in\F_{2^n}} \chi_1\left(b x^{-2^k-1}+d'\right) \sum_{\beta\in \F_{2^n}} 
\chi_1\left( d'x(\beta^{2^{-k}}+\beta) \right)\\
&=2^e+ \frac{2^{e}}{2^n} 2^{-e}   \sum_{d'\in \F_{2^e}} \sum_{x\in\F_{2^n}} \chi_1\left(b x^{-2^k-1}+d'\right) \sum_{\beta\in \F_{2^n}} 
\chi_1\left(\left((d'x)^{2^k}+d'x\right)\beta \right)\\
&\qquad \text{\small(since  $\Tr\left( d'x(\beta^{2^{-k}}+\beta) \right)=\Tr\big(((d'x)^{2^k}+d'x)\beta \big)=\Tr(d'(x^{2^k}+x)\beta)$)}\\
&=2^e+ \frac{2^{e}}{2^n} 2^{n-e}    \sum_{\substack{d'\in \F_{2^e},x\in\F_{2^n}\\
d'(x^{2^k}+x)=0 }} \chi_1\left(b x^{-2^k-1}+d'\right) \\
&=2^e+ \frac{2^{e}}{2^n} 2^{n-e} \sum_{\substack{d'\in \F_{2^e}^*,x\in\F_{2^e}}} \chi_1(bx^{-2}+d')+ \sum_{x \in \F_{2^n}} \chi_1(bx^{-2^k-1})\\
&=2^e+ \frac{2^{e}}{2^n} 2^{n-e} \sum_{d'\in \F_{2^e}^*,x\in\F_{2^e}} \chi_1(bx^{-2}+d')\\
&=2^e- 2^{e}  \delta_0\left(\Tr_e\left(b^{\frac{1}{2}}\right)\right),
\end{align*}
where $\delta_0$ is  the Dirac symbol, defined by $\delta_0(c)=1$, if $c=0$, and 0, otherwise. Thus, $_1\cB_F(1,b) \in\{0,2^e\}$, and the claim of our theorem is shown.
\end{proof}

\textbf{Case 2}: $n/e$ is even.
\begin{enumerate}
\item If $\alpha = \beta$ and $\beta \in  \F_{2^e}$, then
\[T_b^{[1]} = q^2\,\chi_1(\beta)^2 = q^2.\]

\item If $\alpha = \beta$ and $\beta \in \F_{2^n} \backslash \F_{2^e}$, then 
\[T_b^{[2]} = 0.\]

\item If $\alpha \neq \beta$ and $\beta \in \F_{2^e}$, then 
\begin{equation*}
T_b^{[3]} =
\begin{cases}
2^n &~\mbox{if}~  A \neq g^{t(2^e+1)}~\mbox{for any integer}~ t, \\
2^{n+2e} &~\mbox{if}~ A = g^{t(2^e+1)}~\mbox{for some integer}~ t.
\end{cases}
\end{equation*}

\item If $\alpha \neq \beta$ and $\beta \in \F_{2^n} \backslash \F_{2^e}$, then
\begin{enumerate}
 \item If $A \neq g^{t(2^e+1)}$ for any integer $t$, then 
 \[ T_b^{[4(a)]}= 2^n. \] 

 \item If $A = g^{t(2^e+1)}$ for some integer $t$, then
 \begin{enumerate}
  \item If the equation $L_A(x)= B^{2^k}$ is not solvable, where $L_A(x)= A^{2^k}x^{2^{2k}}+Ax$, then 
  \[ T_{b}^{[4(b)(\RNum{1})]}=0.\]
  \item If the equation $L_A(x)= B^{2^k}$ is solvable, then
  \begin{equation*}
T_b^{[4(b)(\RNum{2})]} =
\begin{cases}
2^n &~\mbox{if}~ {\rm Tr}_e(A) \neq 0, \\
2^{n+2e} &~\mbox{if}~ {\rm Tr}_e(A) = 0.
\end{cases}
\end{equation*}
\end{enumerate}
\end{enumerate}
\end{enumerate}
Now we shall summarize the above discussion in the following theorem.
\begin{thm}
Let $F(x)=x^{2^k+1}$, $1\leq k < n$ be a function on $\F_{2^n}$, $n\geq2$. Let $c=1$ and $n/e$ be even, where $e=\gcd(k,n)$. Then the $c$-BCT entry $_1\cB_F(1,b)$ of $F$ at $(1,b)$ is given by  
\begin{align*}
2^e+ \frac{1}{2^n} \sum_{(\alpha, \beta)\in \mathcal{G \cup I \cup K}}\chi_1(b(\alpha+\beta))+\frac{2^{2e}}{2^n} \sum_{(\alpha, \beta)\in \mathcal{H \cup L}}\chi_1(b(\alpha+\beta)),
\end{align*}
with $A=\alpha+\beta$, $B=\beta^{2^{n-k}}+\beta$, $L_A(x)= A^{2^k}x^{2^{2k}}+Ax$, and
\begin{align*}
\mathcal{G} = &\{(\alpha, \beta) \in \mathcal{C} \mid A \neq g^{t(2^e+1)}~\mbox{for any integer}~t \},\\
\mathcal{H} = &\{(\alpha, \beta) \in \mathcal{C} \mid A = g^{t(2^e+1)}~\mbox{for some integer}~t \},\\
\mathcal{I} = &\{(\alpha, \beta) \in \mathcal{D} \mid A \neq g^{t(2^e+1)}~\mbox{for any integer}~t \},\\ 
\mathcal{K} = &\{(\alpha, \beta) \in \mathcal{D} \mid A = g^{t(2^e+1)}~\mbox{for some integer}~t,~ {\rm Tr}_e(A)\neq0 , L_A(x) = B^{2^k}~\mbox{is solvable}\},\\
\mathcal{L} = &\{(\alpha, \beta) \in \mathcal{D} \mid A = g^{t(2^e+1)}~\mbox{for some integer}~t,~ {\rm Tr}_e(A)=0, L_A(x) = B^{2^k}~\mbox{is solvable}\}.
\end{align*}
\end{thm}
\begin{proof}
For the proof, we need to define
\[
\mathcal{J} = \{(\alpha, \beta) \in \mathcal{D} \mid A = g^{t(2^e+1)}~\mbox{for an integer}~t, L_A(x) = B^{2^k}~\mbox{is not solvable}\}.
\]
Then 
\begin{equation*}
 \begin{split}
  _1\cB_F(1,b) &= \frac{1}{q^2} \left( \sum_{(\alpha, \beta) \in \mathcal{A}}  \chi_1(b(\alpha+\beta))T_b^{[1]} + \sum_{(\alpha, \beta) \in \mathcal{B}}  \chi_1(b(\alpha+
 \beta))T_b^{[2]} \right. \\
 &\quad \left. +\sum_{(\alpha, \beta) \in \mathcal{G}}  \chi_1(b(\alpha+\beta))T_b^{[3]}+\sum_{(\alpha, \beta) \in \mathcal{H}}  \chi_1(b(\alpha+\beta))T_b^{[3]} \right. \\
 &\quad \left. +\sum_{(\alpha, \beta) \in \mathcal{I}}  \chi_1(b(\alpha+\beta))T_b^{[4(a)}+\sum_{(\alpha, \beta) \in \mathcal{J}}  \chi_1(b(\alpha+\beta))T_b^{[4(b)(\RNum{1})]} \right. \\
 & \quad\left. +\sum_{(\alpha, \beta) \in \mathcal{K}}  \chi_1(b(\alpha+\beta))T_b^{[4(b)(\RNum{2})]} +\sum_{(\alpha, \beta) \in \mathcal{L}}  \chi_1(b(\alpha+\beta))T_b^{[4(b)(\RNum{2})]}\right) \\
 &= \frac{1}{q^2} \left( \sum_{(\alpha, \beta) \in \mathcal{A}} q^2 + 2^n \sum_{(\alpha, \beta) \in \mathcal{G \cup I \cup K}}  \chi_1(b(\alpha+\beta))  + 2^{n+2e} \sum_{(\alpha, \beta) \in \mathcal{H \cup L}}  \chi_1(b(\alpha+\beta))\right)\\
 &= 2^e+ \frac{1}{2^n} \sum_{(\alpha, \beta)\in \mathcal{G \cup I \cup K}}\chi_1(b(\alpha+\beta))+\frac{2^{2e}}{2^n} \sum_{(\alpha, \beta)\in \mathcal{H \cup L}}\chi_1(b(\alpha+\beta)).
 \end{split}
\end{equation*}
This completes the proof.
\end{proof}
\begin{cor}
 Let $F(x)=x^{2^k+1}$, $1\leq k < n$, be a function on $\F_{q}$,   $n\geq2$. Let $c=1$ and $n/e$ be even, where $e=\gcd(k,n)$. With the notations of the previous theorem,  the $c$-boomerang uniformity  of $F$ satisfies
 \[ 
 \beta_{F,c}\leq 2^e +2^{-n}| \mathcal{G \cup I \cup K}|+2^{2e-n} | \mathcal{H \cup L}|.
 \] 
\end{cor}

\section{The case $c \in \F_{2^e}\backslash \{0,1\}.$} \label{S5}

Since the case $c=1$ has already been considered in the previous section, throughout this section we assume that $c\neq1$.
Notice that when $c\in \F_{2^e}^*$, $\beta \in \F_{2^e} \Leftrightarrow \beta c^{-1} \in \F_{2^e}$. Recall that for any fixed $b\neq0$, the $c$-BCT entry is given by,
\[_c\cB_F(1,b) =\frac{1}{q^2}\sum_{\alpha,\beta \in \F_q}\chi_1\left(b\left(\alpha+\beta \right) \right) S_{\alpha, \beta} S_{c\alpha, c^{-1}\beta}.  \]
Let us denote $ \displaystyle T_b =  S_{\alpha,\beta}S_{c\alpha, c^{-1}\beta}$ (we will use superscripts to point out the case we are in, for its value). Recall that $A=\alpha+\beta$ and $B = \beta^{2^{n-k}}+\beta$. Let us denote $\gamma=A^{\frac{1}{2^k+1}}$, $A' = c\alpha+c^{-1}\beta$ and $B'= (c^{-1}\beta)^{2^{n-k}}+ c^{-1}\beta.$ It is easy to observe that the conditions $B=0$ and $B'=0$ are equivalent. Now we shall consider two cases namely, $\frac{n}{e}$ odd and $\frac{n}{e}$ even, respectively.\\
\textbf{Case 1}: $\frac{n}{e}$ is odd.
\begin{enumerate}
 \item Let $A=0, B=0$.
 \begin{enumerate}
  \item If $A'=0, B'=0$, then 
  \[T_b^{[1(a)]}= q^2\chi_1((1+c^{-1})\beta). \]
  \item If $ A'\neq0, B'=0$, then $S_{c\alpha, c^{-1}\beta}= 0$ and hence
 \[T_b^{[1(b)]} = 0.\]
 \end{enumerate} 
 \item Let $A=0, B\neq0$.
 In this case $S_{\alpha,\beta}=0$ and hence
 \[T_b^{[2]} = 0.\]
 \item Let $A\neq0, B=0$.
 Again $S_{\alpha,\beta}=0$ and hence
 $$T_b^{[3]} = 0$$
 \item Let $A\neq0, B\neq0.$
 \begin{enumerate}
  \item Assume $A'=0, B'\neq0$, then $S_{c\alpha,c^{-1}\beta}=0$ and hence
 \[ T_b^{[4(a)]} = 0.\] 
 
 \item Assume  $A'\neq0, B'\neq0.$ In this case, recall that $\gamma^{2^k+1}= A$ and let $\gamma' \in \F_{q}$ such that $(\gamma')^{2^k+1}= A'.$
 \begin{enumerate}
  \item If ${\rm Tr}_e(B\gamma^{-1}) \neq 1$, then $S_{\alpha, \beta} = 0$ and hence \[T_b^{[4(b)(\RNum{1})]}=0.\]
  \item If ${\rm Tr}_e(B\gamma^{-1}) = 1$ and $\rm{Tr}_e(B'(\gamma')^{-1}) \neq 1$, then $S_{c\alpha, c^{-1}\beta} = 0$ and hence 
  \[ T_b^{[4(b)(\RNum{2})]}=0.\]
  \item If ${\rm Tr}_e(B\gamma^{-1}) = 1$ and ${\rm Tr}_e(B'(\gamma')^{-1}) = 1$, then  
  \[ T_b^{[4(b)(\RNum{3})]}= 2^{n+e} \chi_1((1+c^{-1})\beta). \]
 \end{enumerate}
 \end{enumerate}
\end{enumerate}

We now use the above discussion in the following theorem.

\begin{thm}
 Let $F(x)=x^{2^k+1}$, $1\leq k < n$ be a function on $\F_{2^n}$, $n\geq2$. Let $c \in \F_{2^e} \backslash \{0,1\} $ and $n/e$ be odd, where $e=\gcd(k,n)$. Then the $c$-BCT entry $_c\cB_F(1,b)$ of $F$ at $(1,b)$ is given by 
 \begin{equation*}
1+ \frac{2^{e}}{2^n} \sum_{(\alpha,\beta) \in \mathcal{F \cap F'}} \chi_1(b\alpha+(1+c^{-1}+b)\beta)),  
\end{equation*}
where
\begin{align*}
 \mathcal{F} &= \{(\alpha,\beta)\in \F_{q}^2 \mid A,B \neq0 ~and ~ {\rm Tr}_e(B\gamma^{-1}) = 1 \},\\
 \mathcal{F'} &= \{(\alpha,\beta)\in \F_{q}^2 \mid A',B' \neq0 ~and ~ {\rm Tr}_e(B'(\gamma')^{-1}) = 1 \},
\end{align*}
and $A=\alpha+\beta$, $B = \beta^{2^{n-k}}+\beta$, $A' = c\alpha+c^{-1}\beta$ and $B'= (c^{-1}\beta)^{2^{n-k}}+ c^{-1}\beta$,  $\gamma=A^{\frac{1}{2^k+1}}$,  $\gamma'={A'}^{\frac{1}{2^k+1}}$.
\end{thm}
\begin{proof}
 Let 
 \begin{align*}
  \mathcal{A'} &=\{ (\alpha,\beta)\in \F_{q}^2 \mid c\alpha=c^{-1}\beta~and~ c^{-1}\beta \in \F_{2^e}\},\\
  \mathcal{B'} &=\{ (\alpha,\beta)\in \F_{q}^2 \mid c\alpha=c^{-1}\beta~and~ c^{-1}\beta \in \F_q \backslash \F_{2^e}\}, \\
  \mathcal{C'} &=\{ (\alpha,\beta)\in \F_{q}^2 \mid c\alpha \neq c^{-1}\beta~and~ c^{-1}\beta \in \F_{2^e}\}, \\
  \mathcal{D'} &=\{ (\alpha,\beta)\in \F_{q}^2 \mid c\alpha \neq c^{-1}\beta~and~ c^{-1}\beta \in \F_q \backslash \F_{2^e}\}, \\
  \mathcal{E'} &=\{ (\alpha, \beta) \in \mathcal{D'} \mid {\rm Tr}_e(B'(\gamma')^{-1}) \neq 1\}.
 \end{align*}
 Then,
 \begin{equation*}
  \begin{split}
   _c\cB_F(1,b) &= \frac{1}{q^2} \left(\sum_{(\alpha, \beta) \in \mathcal{A \cap A'}}  \chi_1(b(\alpha+\beta))T_b^{[1(a)]} + \sum_{(\alpha, \beta) \in \mathcal{A \cap C'}}  \chi_1(b(\alpha+\beta))T_b^{[1(b)]} \right. \\
   & \left. + \sum_{(\alpha, \beta) \in \mathcal{B}}  \chi_1(b(\alpha+\beta))T_b^{[2]} + \sum_{(\alpha, \beta) \in \mathcal{C}}  \chi_1(b(\alpha+\beta))T_b^{[3]} \right. \\ 
   & \left. + \sum_{(\alpha, \beta) \in \mathcal{D \cap B'}}  \chi_1(b(\alpha+\beta))T_b^{[4(a)]} + \sum_{(\alpha, \beta) \in \mathcal{E}}  \chi_1(b(\alpha+\beta))T_b^{[4(b)(\RNum{1})]} \right. \\
   & \left. + \sum_{(\alpha, \beta) \in \mathcal{F \cap E'}}  \chi_1(b(\alpha+\beta))T_b^{[4(b)(\RNum{2})]} + \sum_{(\alpha, \beta) \in \mathcal{F \cap F'}}  \chi_1(b(\alpha+\beta))T_b^{[4(b)(\RNum{3})]} \right) \\
   &= \sum_{(\alpha, \beta) \in \mathcal{A \cap A'}} \chi_1(b\alpha+ (1+c^{-1}+b)\beta)) \\ &+ \frac{2^e}{2^n} \sum_{(\alpha, \beta) \in \mathcal{F \cap F'}}  \chi_1(b\alpha+ (1+c^{-1}+b)\beta)) \\
  &= 1 + \frac{2^e}{2^n} \sum_{(\alpha, \beta) \in \mathcal{F \cap F'}}  \chi_1(b\alpha+ (1+c^{-1}+b)\beta)).
  \end{split}
 \end{equation*}
This completes the proof.  
\end{proof}
\begin{cor}
 Let $F(x)=x^{2^k+1}$, $1\leq k < n$, be a function on $\F_{q}$,   $n\geq2$. Let $c\in\F_{2^e}\setminus\{0,1\}$ and $n/e$ be odd, where $e=\gcd(k,n)$. With the notations of the previous theorem,  the $c$-boomerang uniformity  of $F$ satisfies
 \[ 
 \beta_{F,c}\leq 1 +2^{e-n}| \mathcal{F  \cap F'}|.
 \] 
\end{cor}

\textbf{Case 2}: $n/e$ is even.
\begin{enumerate}
 \item Let $A=0, B=0.$
 \begin{enumerate}
  \item  If $A'=0, B'=0$, then 
 \[ T_b^{[1(a)]} = \chi_1((1+c^{-1})\beta)~ q^2. \]
 
 \item If $A'\neq0, B'=0,$ let
\[\mathcal{G'} = \{(\alpha, \beta) \in \mathcal{C'} \mid A' \neq g^{t(2^e+1)}~\mbox{for any integer}~t \},\]
\[\mathcal{H'} = \{(\alpha, \beta) \in \mathcal{C'} \mid A' = g^{t(2^e+1)}~\mbox{for some integer}~t \}.\]
Then,
\begin{equation*}
T_b^{[1(b)]} =
\begin{cases}
(-1)^{\frac{m}{e}} 2^{m+n}  \chi_1((1+c^{-1})\beta) &~if~ (\alpha, \beta) \in \mathcal{A \cap G'},\\
(-1)^{\frac{m}{e}+1} 2^{m+n+e}  \chi_1((1+c^{-1})\beta) &~if~ (\alpha, \beta) \in \mathcal{A \cap H'}.
\end{cases}
\end{equation*} 
\end{enumerate}
\item Let $A=0, B\neq0.$\\ 
In this case $S_{\alpha,\beta}=0$ and hence
\[ T_b^{[2]} = 0. \]

\item Let $A\neq0, B=0.$
\begin{enumerate}
\item If $A'=0, B'=0$, then $T_b^{[3(a)]} $ is given by
\begin{equation*}
\begin{cases}
(-1)^{\frac{m}{e}} 2^{m+n}  \chi_1((1+c^{-1})\beta)  &~if~(\alpha, \beta) \in \mathcal{A' \cap G},\\
(-1)^{\frac{m}{e}+1} 2^{m+n+e}  \chi_1((1+c^{-1})\beta)&~if~(\alpha, \beta) \in \mathcal{A' \cap H}.
\end{cases}
\end{equation*}

\item If $A'\neq0, B'=0,$ then 
\begin{equation*}
T_b^{[3(b)]}=
\begin{cases}
2^n \chi_1((1+c^{-1})\beta) &~if~ (\alpha, \beta) \in \mathcal{G \cap G'},\\
-2^{n+e} \chi_1((1+c^{-1})\beta) &~if~ (\alpha, \beta) \in \mathcal{G \cap H'},\\
-2^{n+e} \chi_1((1+c^{-1})\beta) &~if~ (\alpha, \beta) \in \mathcal{H \cap G'},\\
2^{n+2e} \chi_1((1+c^{-1})\beta) &~if~ (\alpha, \beta) \in \mathcal{H \cap H'}.\\
  \end{cases}
\end{equation*}
\end{enumerate}
\item Let $A\neq0, B\neq0.$\\
\begin{enumerate}
 \item If $A'=0, B'\neq0,$ then $S_{c\alpha,c^{-1}\beta}=0$ and hence
 \[ T_b^{[4(a)]} = 0. \]
 
 \item If $A'\neq0, B'\neq0,$ let
 \begin{align*}
\mathcal{I'} = & \{(\alpha, \beta) \in \mathcal{D'} \mid A' \neq g^{t(2^e+1)}~\mbox{for any integer}~t \}, \\
\mathcal{J'} = &\{(\alpha, \beta) \in \mathcal{D'} \mid A' = g^{t(2^e+1)}~\mbox{for some integer}~t,\\ 
& L_{A'}(x) = (B')^{2^k}~\mbox{is not solvable}\},\\ 
\mathcal{K'} = &\{(\alpha, \beta) \in \mathcal{D'} \mid A' = g^{t(2^e+1)}~\mbox{for some integer}~t,\\ 
&{\rm Tr}_e(A')\neq0,  L_{A'}(x) = (B')^{2^k}~\mbox{is solvable}  \}, \\ 
\mathcal{L'} = &\{(\alpha, \beta) \in \mathcal{D'} \mid A' = g^{t(2^e+1)}~\mbox{for some integer}~t,\\ 
& {\rm Tr}_e(A')=0, L_{A'}(x) = (B')^{2^k}~\mbox{is solvable} \}.
\end{align*}
Then, 
\begin{equation*}
T_b^{[4(b)]}=
  \begin{cases}
   2^n\cdot M  &~if~ (\alpha, \beta) \in \mathcal{(I \cup K) \cap (I' \cup K') },\\
   0 &~if~ (\alpha, \beta) \in \mathcal{(I \cup K \cup L) \cap J'},\\
   -2^{n+e}\cdot M &~if~ (\alpha, \beta) \in \mathcal{(I \cup K) \cap L'},\\
   0 &~if~ (\alpha, \beta) \in \mathcal{J \cap (I' \cup J' \cup K' \cup L')},\\
   -2^{n+e} \cdot M &~if~ (\alpha, \beta) \in \mathcal{L \cap (I' \cup K')},\\
   2^{n+2e} \cdot M &~if~ (\alpha, \beta) \in \mathcal{L \cap L'},
  \end{cases}
\end{equation*}
where $M = \chi_1((1+c^{-1})\beta) \chi_1\left(AA'x_{A}^{2^k+1}x_{A'}^{2^k+1}\right)$ and $x_{A}, x_{A'}$ are the solutions of the equations $L_A(x) = B^{2^k}$ and $L_{A'}(x) = (B')^{2^k}$, respectively.
\end{enumerate} 
\end{enumerate}
We now summarize the above discussion in the following theorem.
\begin{thm}
Let $F(x)=x^{2^k+1}$, $1\leq k < n$ be a function on $\F_{2^n}$, $n\geq2$. Let $c \in \F_{2^e} \backslash \{0,1\}$ and $n/e$ be even, where $e=\gcd(k,n)$. With the previous notations,  the $c$-BCT entry $_c\cB_F(1,b)$ of $F$ at $(1,b)$ is given by  
\begin{equation*}
\begin{split}
\frac{1}{q^2}  &\left( \sum_{(\alpha, \beta) \in \mathcal{A \cap A'}} \chi_1 (b (\alpha+\beta))T_b^{[1(a)]} + \sum_{(\alpha, \beta) \in \mathcal{A \cap G'}} \chi_1 (b (\alpha+\beta))T_b^{[1(b)]} \right. \\
&\left. + \sum_{(\alpha, \beta) \in \mathcal{A \cap H'}} \chi_1 (b (\alpha+\beta))T_b^{[1(b)]} + \sum_{(\alpha, \beta) \in \mathcal{A' \cap G}} \chi_1 (b (\alpha+\beta))T_b^{[3(a)]} \right. \\
&\left. + \sum_{(\alpha, \beta) \in \mathcal{A' \cap H}} \chi_1 (b (\alpha+\beta))T_b^{[3(a)]} + \sum_{(\alpha, \beta) \in \mathcal{G \cap G'}} \chi_1 (b (\alpha+\beta))T_b^{[3(b)]} \right. \\ &\left. +\sum_{(\alpha, \beta) \in \mathcal{G \cap H'}} \chi_1 (b (\alpha+\beta))T_b^{[3(b)]} + \sum_{(\alpha, \beta) \in \mathcal{H \cap G'}} \chi_1 (b (\alpha+\beta))T_b^{[3(b)]} \right. \\ &\left. + \sum_{(\alpha, \beta) \in \mathcal{H \cap H'}} \chi_1 (b (\alpha+\beta))T_b^{[3(b)]} + + \sum_{(\alpha, \beta) \in \mathcal{(I \cup K) \cap (I' \cup K') }} \chi_1 (b (\alpha+\beta))T_b^{[4(b)]} \right. \\ 
&\left. + \sum_{(\alpha, \beta) \in \mathcal{(I \cup K) \cap L'}} \chi_1 (b (\alpha+\beta))T_b^{[4(b)]} + \sum_{(\alpha, \beta) \in \mathcal{L \cap (I' \cup K')}} \chi_1 (b (\alpha+\beta))T_b^{[4(b)]} \right. \\ 
&\left. + \sum_{(\alpha, \beta) \in \mathcal{L \cap L'}} \chi_1 (b (\alpha+\beta))T_b^{[4(b)]}\right).
\end{split}
\end{equation*}
\end{thm}

\section{The general case} \label{S6}

Since the case $c\in \F_{2^e}$ has already been considered in previous sections, throughout this section we assume that $c\in \F_{2^n} \backslash \F_{2^e}$. Recall that for any fixed $b\neq0$, the $c$-BCT entry is given by,
\[_c\cB_F(1,b) =\frac{1}{q^2}\sum_{\alpha,\beta \in \F_q}\chi_1\left(b\left(\alpha+\beta \right) \right) S_{\alpha, \beta} S_{c\alpha, c^{-1}\beta}.  \]
Let us denote $ \displaystyle T_b =  S_{\alpha,\beta}S_{c\alpha, c^{-1}\beta}$. Recall that $A=\alpha+\beta$, $B = \beta^{2^{n-k}}+\beta$, $A' = c\alpha+c^{-1}\beta$ and $B'= (c^{-1}\beta)^{2^{n-k}}+ c^{-1}\beta.$ Notice that, when $c\in \F_{2^n} \backslash \F_{2^e}$ then $\beta \in \F_{2^e}^*$, and so,  $\beta c^{-1} \in \F_{2^n}\backslash \F_{2^e}$, otherwise $c\in \F_{2^e}$. Thus $B=0=B'$ if and only if $\beta=0$. Also, observe that the conditions $A=0=A'$ if and only if $\alpha=0=\beta$. Now we shall consider two cases namely, $\frac{n}{e}$ is odd and $\frac{n}{e}$ is even, respectively.\\
\textbf{Case 1}: $\frac{n}{e}$ is odd.

\begin{enumerate}
 \item Let $A=0, B=0.$\\
 Notice that the cases $A'=0, B' \neq 0$, and $A'\neq 0, B' =0$ would not arise, therefore, we shall calculate $T_b$ in remaining two cases only.
 \begin{enumerate}
  \item If $A'=0, B' = 0$, then
 \[ T_b^{[1(a)]} = \chi_1((1+c^{-1})\beta)\, q^2 .\]
  \item If $A'\neq0, B'\neq0$, then 
  \begin{equation*}
  T_b^{[1(b)]} = 
  \begin{cases}
   0 &~if~{\rm Tr}_e (B'(\gamma')^{-1}) \neq 1,\\
   \left(\frac{2}{n/e} \right)^e 2^{\frac{3n+e}{2}}  \chi_1((1+c^{-1})\beta) &~if~{\rm Tr}_e (B'(\gamma')^{-1}) = 1.
  \end{cases}
 \end{equation*}
 \end{enumerate} 
 \item Let $A=0, B\neq0.$
 In this case $S_{\alpha, \beta}=0$ and hence
 $$T_b^{[2]} = 0.$$
 \item Let $A\neq0, B=0.$
 Again, $S_{\alpha, \beta}=0$ and hence
 $$T_b^{[3]} = 0.$$
 \item Let $A\neq0, B\neq0.$
 \begin{enumerate}
  \item If $A'=0, B'=0$, then 
  \begin{equation*}
  T_b^{[4(a)]} = 
  \begin{cases}
   0 &~if~{\rm Tr}_e (B\gamma^{-1}) \neq 1,\\
   \left(\frac{2}{n/e} \right)^e 2^{\frac{3n+e}{2}}  \chi_1((1+c^{-1})\beta) &~if~{\rm Tr}_e (B\gamma^{-1}) = 1.
  \end{cases}
 \end{equation*}
\item If $A' = 0, B' \neq 0$, then $S_{c\alpha, c^{-1}\beta}=0$ and hence
 $$T_b^{[4(b)]}=0.$$
\item If $A' \neq 0, B' = 0$, then again $S_{c\alpha, c^{-1}\beta}=0$ and hence
 $$T_b^{[4(c)]}=0.$$
 \item If $A'\neq0, B'\neq0$, then the only relevant case is  and 
 
 \begin{equation*}
  T_b^{[4(d)]} = 
  \begin{cases}
    2^{n+e}  \chi_1((1+c^{-1})\beta) &~if~(\alpha, \beta) \in \mathcal{F \cap F'},\\
    0 &~otherwise.
  \end{cases}
 \end{equation*}
 \end{enumerate}
\end{enumerate}
We now summarize the above discussion in the following theorem.

\begin{thm}
 Let $F(x)=x^{2^k+1}$, $1\leq k < n$ be a function on $\F_{2^n}$, $n\geq2$. Let $c \in \F_{2^n} \backslash \F_{2^e}$ and $n/e$ be odd, where $e=\gcd(k,n)$. Then the $c$-BCT entry $_c\cB_F(1,b)$ of $F$ at $(1,b)$ is given by 
\begin{equation*}
\begin{split}
1+ \frac{2^{\frac{e}{2}}}{2^n} \sum_{(\alpha, \beta)\in \mathcal{(A \cap F') \cup (A' \cap F)}} \chi_1(b\alpha+(1+c^{-1}+b)\beta)) \\
+ \frac{2^e}{2^n} \sum_{(\alpha, \beta)\in \mathcal{F \cap F'}} \chi_1(b\alpha+(1+c^{-1}+b)\beta)).
\end{split}
\end{equation*}
\end{thm}
\begin{proof}
 \begin{equation*}
\begin{split}
_c\cB_F(1,b) = \frac{1}{q^2} &\left( \sum_{(\alpha, \beta)\in \mathcal{A \cap A'}} \chi_1(b(\alpha+\beta))T_b^{[1(a)]}+ \sum_{(\alpha, \beta)\in \mathcal{A \cap F'}} \chi_1(b(\alpha+\beta))T_b^{[1(b)]} \right. \\ 
& \left. + \sum_{(\alpha, \beta)\in \mathcal{F \cap A'}} \chi_1(b(\alpha+\beta))T_b^{[4(a)]}+ \sum_{(\alpha, \beta)\in \mathcal{F \cap F'}} \chi_1(b(\alpha+\beta))T_b^{[4(d)]} \right) \\
= 1+ & \left(\frac{2}{n/e} \right)^e \cdot 2^{\frac{e-n}{2}} \sum_{(\alpha, \beta)\in \mathcal{(A \cap F') \cup (A' \cap F)}} \chi_1(b\alpha+(1+c^{-1}+b)\beta))  \\ 
&  + 2^{e-n} \sum_{(\alpha, \beta)\in \mathcal{F \cap F'}} \chi_1(b\alpha+(1+c^{-1}+b)\beta)). 
\end{split}
\end{equation*}
\end{proof}
\begin{cor}
 Let $F(x)=x^{2^k+1}$, $1\leq k < n$, be a function on $\F_{q}$,   $n\geq2$. Let $c \in \F_{2^n} \backslash \F_{2^e}$ and $n/e$ be odd, where $e=\gcd(k,n)$.  With the notations of the previous theorem,  the $c$-boomerang uniformity  of $F$ satisfies
 \[ 
 \beta_{F,c}\leq 1 + \left(\frac{2}{n/e} \right)^e \cdot 2^{\frac{e-n}{2}}|\mathcal{(A \cap F') \cup (A' \cap F)}|+2^{e-n} |  \mathcal{F \cap F'}|.
 \] 
\end{cor}

\textbf{Case 2}: $n/e$ is even.

\begin{enumerate}
 \item Let $A=0, B=0.$
 Notice that the cases $A'=0, B' \neq 0$, and $A'\neq 0, B' =0$ would not arise, therefore, we shall calculate $T_b$ in remaining two cases only.
 \begin{enumerate}
  \item If $A'=0, B' = 0$, then
 \[ T_b^{[1(a)]} = \chi_1((1+c^{-1})\beta)\, q^2 .\]
 \item If $A' \neq0, B' \neq0 $, then
 \begin{equation*}
  T_b^{[1(b)]} = 
  \begin{cases}
  (-1)^{\frac{m}{e}} 2^{m+n} M'  &~if~ (\alpha, \beta) \in \mathcal{A \cap (I' \cup K')},\\
  0 &~if~ (\alpha, \beta) \in \mathcal{A \cap J'},\\
  (-1)^{\frac{m}{e}+1} 2^{m+n+e}M'  &~if~ (\alpha, \beta) \in \mathcal{A \cap L'},   
  \end{cases}
 \end{equation*}
 where $M' = \chi_1((1+c^{-1})\beta) \chi_1(A'x_{A'}^{2^k+1}).$
\end{enumerate} 
 
 \item Let $A=0, B\neq0.$
 In this case $S_{\alpha,\beta}=0$ and hence
 $$T_b^{[2]} = 0$$
 \item Let $A\neq0, B=0.$
 Notice that the case $A'=0, B'=0$ would not arise. Now we shall calculate $T_b$ in the remaining cases.
 \begin{enumerate}
  \item If $A'=0, B'\neq0,$ then $S_{c\alpha, c^{-1}\beta}=0$ and hence \[ T_b^{[3(a)]}=0.\]
  \item If $A'\neq 0, B'=0,$ then
  \begin{equation*}
   T_b^{[3(b)]} =
   \begin{cases}
   2^{n} \chi_1((1+c^{-1})\beta) &~if~ (\alpha, \beta) \in \mathcal{G \cap G'},\\
   -2^{n+e} \chi_1((1+c^{-1})\beta) &~if~ (\alpha, \beta) \in \mathcal{G \cap H'}, \\
   -2^{n+e} \chi_1((1+c^{-1})\beta) &~if~ (\alpha, \beta) \in \mathcal{H \cap G'}, \\
   2^{n+2e} \chi_1((1+c^{-1})\beta) &~if~ (\alpha, \beta) \in \mathcal{H \cap H'}.\\
   \end{cases}
  \end{equation*}
  \item If $A' \neq 0, B' \neq 0$, then
  \begin{equation*}
  T_b^{[3(c)]} = 
  \begin{cases}
  2^{n} M' &~if~ (\alpha, \beta) \in \mathcal{G \cap (I' \cup K')},\\
  0 &~if~ (\alpha, \beta) \in \mathcal{(G \cup H) \cap J'},\\
  -2^{n+e} M' &~if~ (\alpha, \beta) \in \mathcal{G \cap L'},\\
-2^{n+e} M' &~if~ (\alpha, \beta) \in \mathcal{H \cap(I' \cup K')},\\
2^{n+2e} M' &~if~ (\alpha, \beta) \in \mathcal{H \cap L')}.\\
  \end{cases}
 \end{equation*}  
\end{enumerate}
\item Let $A\neq0, B\neq0.$
\begin{enumerate}
 \item If $A'=0, B'=0,$ then
 \begin{equation*}
  T_b^{[4(a)]} = 
  \begin{cases}
  (-1)^{\frac{m}{e}} 2^{m+n} M'' &~if~ (\alpha, \beta) \in \mathcal{A' \cap (I \cup K)},\\
  0 &~if~ (\alpha, \beta) \in \mathcal{A' \cap J},\\
  (-1)^{\frac{m}{e}+1} 2^{m+n+e} M'' &~if~ (\alpha, \beta) \in \mathcal{A' \cap L} ,  
  \end{cases}
 \end{equation*}
 where $M''=\chi_1((1+c^{-1})\beta) \chi_1(Ax_{A}^{2^k+1}).$
 \item If $A'=0, B'\neq0$, then $S_{c\alpha, c^{-1}\beta}=0$ and hence
 \[T_b^{[4(b)]}=0.\]
 \item If $A'\neq0, B'=0$, then
 \allowdisplaybreaks
 \begin{equation*}
  T_b^{[4(c)]} = 
  \begin{cases}
  2^{n} M'' &~if~ (\alpha, \beta) \in \mathcal{G' \cap (I \cup K)},\\
  0 &~if~ (\alpha, \beta) \in \mathcal{(G' \cup H') \cap J},\\
  -2^{n+e} M'' &~if~ (\alpha, \beta) \in \mathcal{G' \cap L},\\
-2^{n+e} M'' &~if~ (\alpha, \beta) \in \mathcal{H' \cap(I \cup K)},\\
2^{n+2e} M'' &~if~ (\alpha, \beta) \in \mathcal{H' \cap L}.\\
  \end{cases}
 \end{equation*} 
 \item If $A' \neq0, B'\neq0$, then
  \allowdisplaybreaks
 \begin{equation*}
T_b^{[4(d)]}=
  \begin{cases}
   2^n  M''' &~if~ (\alpha, \beta) \in \mathcal{(I \cup K) \cap (I' \cup K') },\\
   0 &~if~ (\alpha, \beta) \in \mathcal{(I \cup K \cup L) \cap J'},\\
   -2^{n+e} M''' &~if~ (\alpha, \beta) \in \mathcal{(I \cup K) \cap L'},\\
   0 &~if~ (\alpha, \beta) \in \mathcal{J \cap (I' \cup J' \cup K' \cup L')},\\
   -2^{n+e} M''' &~if~ (\alpha, \beta) \in \mathcal{L \cap (I' \cup K')},\\
   2^{n+2e} M''' &~if~ (\alpha, \beta) \in \mathcal{L \cap L'},
  \end{cases}
\end{equation*}
where $M'''=\chi_1((1+c^{-1})\beta) \chi_1(Ax_{A}^{2^k+1}+A'x_{A'}^{2^k+1}).$ 
\end{enumerate}
\end{enumerate}
We now summarize the above discussion in the form of following theorem.
\begin{thm}
Let $F(x)=x^{2^k+1}$, $1\leq k < n$ be a function on $\F_{2^n}$, $n\geq2$. Let $c\in \F_{2^n}\backslash \F_{2^e}$ and $n/e$ be even, where $e=\gcd(k,n)$. With the prior notations, the $c$-BCT entry $_c\cB_F(1,b)$ of $F$ at $(1,b)$ is given by  
 \allowdisplaybreaks
\begin{align*}
\frac{1}{q^2} &\left( \sum_{(\alpha, \beta) \in \mathcal{A \cap A'}} \chi_1 (b (\alpha+\beta))T_b^{[1(a)]} + \sum_{(\alpha, \beta) \in \mathcal{A \cap (I' \cup K')}} \chi_1 (b (\alpha+\beta))T_b^{[1(b)]} \right. \\ 
& \left. +\sum_{(\alpha, \beta) \in \mathcal{A \cap L'}} \chi_1 (b (\alpha+\beta))T_b^{[1(b)]} + \sum_{(\alpha, \beta) \in \mathcal{G \cap G'}} \chi_1 (b (\alpha+\beta))T_b^{[3(b)]} \right. \\
&\left. +\sum_{(\alpha, \beta) \in \mathcal{G \cap H'}} \chi_1 (b (\alpha+\beta))T_b^{[3(b)]} + \sum_{(\alpha, \beta) \in \mathcal{H \cap G'}} \chi_1 (b (\alpha+\beta))T_b^{[3(b)]} \right. \\
& \left. +\sum_{(\alpha, \beta) \in \mathcal{H \cap H'}} \chi_1 (b (\alpha+\beta))T_b^{[3(b)]} + \sum_{(\alpha, \beta) \in \mathcal{G \cap(I' \cup K')}} \chi_1 (b (\alpha+\beta))T_b^{[3(c)]} \right. \\
& \left. + \sum_{(\alpha, \beta) \in \mathcal{G \cap L'}} \chi_1 (b (\alpha+\beta))T_b^{[3(c)]} +  \sum_{(\alpha, \beta) \in \mathcal{H \cap (I' \cup K') }} \chi_1 (b (\alpha+\beta))T_b^{[3(c)]} \right. \\
& \left. + \sum_{(\alpha, \beta) \in \mathcal{H \cap L'}} \chi_1 (b (\alpha+\beta))T_b^{[3(c)]} + \sum_{(\alpha, \beta) \in \mathcal{A' \cap (I \cup K)}} \chi_1 (b (\alpha+\beta))T_b^{[4(a)]} \right. \\
& \left. + \sum_{(\alpha, \beta) \in \mathcal{A' \cap L}} \chi_1 (b (\alpha+\beta))T_b^{[4(a)]}+
\sum_{(\alpha, \beta) \in \mathcal{G' \cap (I\cup K) }} \chi_1 (b (\alpha+\beta))T_b^{[4(c)]} \right. \\
& \left. + \sum_{(\alpha, \beta) \in \mathcal{G' \cap L}} \chi_1 (b (\alpha+\beta))T_b^{[4(c)}+ \sum_{(\alpha, \beta) \in \mathcal{H' \cap (I \cup K) }} \chi_1 (b (\alpha+\beta))T_b^{[4(c)]} \right. \\
& \left. + \sum_{(\alpha, \beta) \in \mathcal{H' \cap L}} \chi_1 (b (\alpha+\beta))T_b^{[4(c)]} + \sum_{(\alpha, \beta) \in \mathcal{(I \cup K) \cap (I' \cup K')}} \chi_1 (b (\alpha+\beta))T_b^{[4(d)]} \right. \\
& \left. + \sum_{(\alpha, \beta) \in \mathcal{(I \cup K) \cap L'}} \chi_1 (b (\alpha+\beta))T_b^{[4(d)]}+ \sum_{(\alpha, \beta) \in \mathcal{(I' \cup K') \cap L}} \chi_1 (b (\alpha+\beta))T_b^{[4(d)]} \right. \\
& \left. + \sum_{(\alpha, \beta) \in \mathcal{L' \cap L}} \chi_1 (b (\alpha+\beta))T_b^{[4(d)]} \right).
 \end{align*}
\end{thm}
\section{discussion on equivalence}\label{S7}
Boura and Canteaut~\cite{BC} showed that the BCT table is preserved under the affine equivalence but not under the extended affine equivalence (and consequently under the CCZ-equivalence). It is quite natural to ask a similar question in the context of $c$-BCT. It is straightforward to see that in the case of even characteristic, $c$-BCT and $c^{-1}$-BCT entries of an $(n,n)$-function $F: \F_{2^n} \to \F_{2^n}$ are the same under the transformations $x \mapsto x+a$ and $y \mapsto y+a$, since the $c$-boomerang system
\begin{align*}
 \begin{cases}
  F(x)+cF(y)=b \\
  F(x+a)+c^{-1}F(y+a)=b
 \end{cases}
\end{align*}
becomes 
\begin{align*}
 \begin{cases}
  F(x)+c^{-1}F(y)=b \\
  F(x+a)+cF(y+a)=b.
 \end{cases}
\end{align*}
We consider the binomial $G(x)=x^{2^k+1}+ux^{2^{n-k}+1} \in \F_{2^n}[x]$, which is a PP if and only if $\displaystyle \frac{n}{e}$ is odd and $\displaystyle u\neq g^{t(2^e-1)}$, where $e=\gcd(n,k)= \gcd(n-k,k)$ and $g$ is the primitive element of $\F_{2^n}.$ Notice that $G(x)= (L \circ F)(x)$ where $L(x)=x^{2^k}+ux$ and $F(x)= x^{2^{n-k}+1}.$ When $n=6, k=2$ and $u=g$, where $g$ is a root of the primitive polynomial $y^6 + y^4 + y^3 + y + 1$ over $\F_{2}$, then $L(x)$ and $G(x)$ are PP. It is easy to see from the Table~\ref{T1} in the Appendix~\ref{app}
that the $c$-BCT is not preserved under the (output applied) affine equivalence. However, if the affine transformation is applied to the input, that is, $G(x)=(F\circ L)(x)$, then the $c$-BCT spectrum is preserved, as was the case for the $c$-differential uniformity.

\section*{Acknowledgements}
Sartaj Ul Hasan is partially supported by MATRICS grant MTR/2019/000744 from the Science and Engineering Research Board, Government of India.

\appendix 
\section{} \label{app}
Let $G(x)=x^5+gx^{17} = (x^4+gx) \circ x^{17} \in \F_{2^6}[x],$  where $g$ is a root of the primitive polynomial $y^6 + y^4 + y^3 + y + 1$ over $\F_{2}$. The following Table~\ref{T1} gives the set of the $c$-BCT entries for $x^{17}$ as well as $G(x)=x^5+gx^{17}$ for all $c\in \F_{2^n}\backslash \F_{2}.$ In view of the discussion in Section~\ref{S7}, it is sufficient to compute the set of $c$-BCT entries for either of $c$ or $c^{-1}$ as they are going to be exactly the same.

 \begin{table} [h!]
 \begin{center}

 \begin{tabular}{ |c|c|c| } 
 
\hline
 $c$ & Set of $c$-BCT entries of $x^{17}$ & Set of $c$-BCT entries of $G(x)$ \\
 \hline
 $g$ & \{0,1,2,3,4\} & \{0,1,2,3,4,5\} \\ 
 \hline 
 $g^2$ & \{0,1,2,3,4\} & \{0,1,2,3,4,5\} \\ 
 \hline
 $g^3$ & \{0,1,2,3,5\} & \{0,1,2,3,4,5\} \\ 
 \hline
 $g^4$ & \{0,1,2,3,4\} & \{0,1,2,3,4,5\} \\ 
 \hline
 $g^5$ & \{0,1,2,3,4\} & \{0,1,2,3,4,5,6\} \\ 
 \hline
 $g^6$ & \{0,1,2,3,5\} & \{0,1,2,3,4,5,6\} \\ 
 \hline
 $g^7$ & \{0,1,2,3,4,5\} & \{0,1,2,3,4,5,6\} \\ 
 \hline
 $g^8$ & \{0,1,2,3,4\} & \{0,1,2,3,4,5\} \\ 
 \hline
 $g^9$ & \{0,1,2,3,4\} & \{0,1,2,3,4,5,6\} \\ 
 \hline
 $g^{10}$ & \{0,1,2,3,4\} & \{0,1,2,3,4,5\} \\ 
 \hline 
 $g^{11}$ & \{0,1,2,3\} & \{0,1,2,3,4,5\} \\ 
 \hline 
 $g^{12}$ & \{0,1,2,3,5\} & \{0,1,2,3,4,5\} \\ 
 \hline
 $g^{13}$ & \{0,1,2,3\} & \{0,1,2,3,4,5\} \\ 
 \hline
 $g^{14}$ & \{0,1,2,3,4,5\} & \{0,1,2,3,4,5,6\} \\ 
 \hline 
 $g^{15}$ & \{0,1,2,3,5\} & \{0,1,2,3,4,5\} \\ 
 \hline
 $g^{16}$ & \{0,1,2,3,4\} & \{0,1,2,3,4,5\} \\ 
 \hline 
 $g^{17}$ & \{0,1,2,3,4\} & \{0,1,2,3,4,5,6\} \\ 
 \hline 
 $g^{18}$ & \{0,1,2,3,4\} & \{0,1,2,3,4,5,6\} \\ 
 \hline
 $g^{19}$ & \{0,1,2,3\} & \{0,1,2,3,4,5\} \\ 
 \hline
 $g^{20}$ & \{0,1,2,3,4\} & \{0,1,2,3,4,5,6\} \\ 
 \hline
 $g^{21}$ & \{0,1,4\} & \{0,1,4\} \\ 
 \hline
 $g^{22}$ & \{0,1,2,3\} & \{0,1,2,3,4,5\} \\ 
 \hline
 $g^{23}$ & \{0,1,2,3,4\} & \{0,1,2,3,4,5\} \\ 
 \hline
 $g^{24}$ & \{0,1,2,3,5\} & \{0,1,2,3,4,5,6\} \\ 
 \hline 
 $g^{25}$ & \{0,1,2,3\} & \{0,1,2,3,4,5\} \\ 
 \hline
 $g^{26}$ & \{0,1,2,3\} & \{0,1,2,3,4,5\} \\ 
 \hline
 $g^{27}$ & \{0,1,2,3,4\} & \{0,1,2,3,4,5,6\} \\ 
 \hline
 $g^{28}$ & \{0,1,2,3,4,5\} & \{0,1,2,3,4,5,6\} \\ 
 \hline
 $g^{29}$ & \{0,1,2,3,4\} & \{0,1,2,3,4,5\} \\ 
 \hline
 $g^{30}$ & \{0,1,2,3,5\} & \{0,1,2,3,4,5,6\} \\ 
 \hline
 $g^{31}$ & \{0,1,2,3,4\} & \{0,1,2,3,4,5\} \\ 
 \hline
\end{tabular}
\caption{$c$-BCT entries of $x^{17}$ and $x^{5}+gx^{17}$ } \label{T1}
\end{center}
\end{table}

\begin{thebibliography}{99}
\bibitem{BC1} D. Bartoli, M. Calderini, {\em On construction and (non)existence of $c$-(almost) perfect nonlinear functions}, Finite Fields Appl., 72 (2021), 101835.


\bibitem{BC} C. Boura, A. Canteaut, { \em On the boomerang uniformity of cryptographic Sboxes}, IACR Trans. Symmetric Cryptol., vol. 2018, no. 3, 290--310, 2018.

 \bibitem{CH1} C.~Carlet, {\em Boolean functions for cryptography and error correcting codes}, In: Y. Crama, P. Hammer  (eds.), Boolean Methods and Models,
Cambridge Univ. Press, Cambridge, pp. 257--397, 2010.

\bibitem{cid} C. Cid, T. Huang, T. Peyrin, Y. Sasaki, and L. Song, {\it Boomerang
connectivity table: a new cryptanalysis tool}. In: Nielsen J., Rijmen V. (eds.), Advances in Cryptology-EUROCRYPT 2018, LNCS 10821, Springer, Cham,   pp. 683--714, 2018.

\bibitem{RSC} R. S. Coulter, {\it On the evaluation of a class of Weil sums in characteristic 2}, New Zealand J. of Math., vol. 28, pp. 171--184, 1999.

\bibitem{CS17} T. W.~Cusick, P.~St\u anic\u a,
{Cryptographic Boolean Functions and Applications} (Ed. 2), Academic Press, San Diego, CA,  2017.

\bibitem{CDU} P. Ellingsen, P. Felke, C. Riera, P. St\u anic\u a, A. Tkachenko, {\it C-differentials, multiplicative uniformity and (almost) perfect $c$-nonlinearity}, IEEE Trans. Inform. Theory 66(9) (2020), 5781--5789.

\bibitem{HPRS} S. U. Hasan, M. Pal, C. Riera, P. St\u anic\u a, {\it On the $c$-differential uniformity of certain maps over finite fields}, Des. Codes Cryptogr. 89(2) (2021), 221--239. 

%\bibitem{LN97}
%R. Lidl, H. Niederreiter, FiniteFields (Ed. 2), Encycl. Math. Appl., vol.20, Cambridge Univ. Press, Cambridge, 1997.

\bibitem{KL} K. Li, L. Qu, B. Sun, C. Li, {\it New results about the boomerang uniformity of permutation polynomials,} IEEE Trans. Inform. Theory  65(11) (2019), 7542--7553.
 
 \bibitem{YMZ} S. Mesnager, C. Riera, P. St\u anic\u a, H. Yan and Z. Zhou, {\it Investigations on $c$-(almost) perfect nonlinear functions}, IEEE Trans. Inform. Theory (2021), \url{https://doi.org/10.1109/TIT.2021.3081348}.

\bibitem{NY} K. Nyberg,{ \it  Differentially uniform mappings for cryptography.} In: Helleseth T. (eds.), Advances in Cryptology--EUROCRYPT 1993, LNCS 765, Springer, Berlin, Heidelberg,  pp. 55--64, 1994.


  \bibitem{PSB}
 P. St\u anic\u a, {\em Investigations on $c$-boomerang uniformity and perfect nonlinearity}, 
Discrete Appl. Math., Vol. 304, 297--314, 2021.


\bibitem{PSS} P. St\u anic\u a, {\it Low $c$-differential and $c$-boomerang uniformity of the swapped inverse function,} Discrete Math. 344(10) (2021), 112543.

\bibitem{PSW} P. St\u anic\u a, {\it Using double Weil sums in finding the $c$-boomerang connectivity table for monomial functions on finite fields,}  Appl. Algebra Eng. Commun. Comput., (2021). {\url{https://doi.org/10.1007/s00200-021-00520-9}}

%\bibitem{PSB} P. St\u anic\u a, {\it  Investigations on c-boomerang uniformity and perfect nonlinearity,} \url{ https://arxiv.org/abs/2004.11859}, 2020.
%
%\bibitem{PSS} 
%P. St\u anic\u a, 
%{\em Low $c$-differential and $c$-boomerang uniformity of the swapped inverse function}, to appear in Discrete Mathematics Journal.
%
%\bibitem{PSW} P. St\u anic\u a, { \em Using double Weil sums in finding the Boomerang and the c-Boomerang Connectivity Table for monomial functions on finite fields}, \url{https://arxiv.org/abs/2007.09553}, 2020.

\bibitem{SG20} P. St\u anic\u a, A. Geary, {\em The $c$-differential behaviour of the inverse function under the $EA$-equivalence}, Cryptogr. Commun. 13 (2021), 295--306.
 
 \bibitem{SRT} P. St\u anic\u a, C. Riera, A. Tkachenko,  {\em Characters, Weil sums and c-differential uniformity with an application to the perturbed Gold function}, Cryptogr. Commun. (2021), \url{https://doi.org/10.1007/s12095-021-00485-z}.
  
\bibitem{Wag} D. Wagner, {\it The boomerang attack,} In: L. R. Knudsen (ed.) Fast Software Encryption-FSE 1999. LNCS 1636, Springer, Berlin, Heidelberg,  pp. 156--170, 1999.


 \bibitem{ZH} Z. Zha, L. Hu, {\it Some classes of power functions with low $c$-differential uniformity over finite fields}, Des. Codes Cryptogr. (2021), \url{https://doi.org/10.1007/s10623-021-00866-8}.
 

\end{thebibliography}
\end{document}